\documentclass[letterpaper, 10 pt, conference]{ieeeconf}

\IEEEoverridecommandlockouts
\usepackage{amsmath}
\usepackage{amssymb}
\usepackage{amsthm}
\usepackage{mathtools}

\usepackage[noadjust]{cite}

\newtheorem{theorem}{Theorem}
\newtheorem{lemma}{Lemma}
\newtheorem{proposition}{Proposition}

\theoremstyle{definition}
\newtheorem{definition}{Definition}
\newtheorem{assumption}{Assumption}
\newtheorem{remark}{Remark}
\newtheorem{problem}{Problem}

\newcommand{\argmin}{\operatornamewithlimits{arg\,min}}
\newcommand{\argmax}{\operatornamewithlimits{arg\,max}}
\newcommand{\subjectto}{\operatorname{subject~to}}
\newcommand{\R}{\mathbb{R}}
\newcommand{\Z}{\mathbb{Z}}
\newcommand{\N}{\mathbb{N}}

\newcommand{\mc}[1]{\mathcal{#1}}
\newcommand{\T}{^\top}

\DeclarePairedDelimiter{\curly}{\{}{\}}
\DeclarePairedDelimiter{\norm}{\lVert}{\rVert}

\usepackage[x11names]{xcolor}
\definecolor{pastelMagenta}{HTML}{FF48CF}
\definecolor{pastelPurple}{HTML}{8770FE}
\definecolor{pastelBlue}{HTML}{1BA1EA}
\definecolor{pastelSeaGreen}{HTML}{14B57F}
\definecolor{pastelGreen}{HTML}{3EAA0D}
\definecolor{pastelOrange}{HTML}{C38D09}
\definecolor{pastelRed}{HTML}{F5615C}

\definecolor{myBlue1}{RGB}{49, 114, 174}
\definecolor{myRed1}{RGB}{224, 107, 97}
\definecolor{myGreen1}{RGB}{68, 156, 118}
\definecolor{myPurple1}{RGB}{117, 112, 173}
\definecolor{myYellow1}{RGB}{221, 162, 66}
\definecolor{myMagenta1}{RGB}{202, 98, 159}
\definecolor{myCyan1}{RGB}{114, 179, 224}

\definecolor{myBlue2}{cmyk}{0.8336, 0.5245, 0.0745, 0}
\definecolor{myRed2}{cmyk}{0.0808, 0.7143, 0.5968, 0.0028}
\definecolor{myGreen2}{cmyk}{0.7402, 0.1786, 0.0667, 0.0021}
\definecolor{myPurple2}{cmyk}{0.6079, 0.5931, 0.0401, 0}
\definecolor{myYellow2}{cmyk}{0.1289, 0.3792, 0.8644, 0.0014}
\definecolor{myMagenta2}{cmyk}{0.1821, 0.7493, 0.0432, 0}
\definecolor{myCyan2}{cmyk}{0.5226, 0.1608, 0.0075, 0}

\definecolor{juliaBlue}{rgb}{0.255, 0.388, 0.847}
\definecolor{juliaRed}{rgb}{0.796, 0.235, 0.2}
\definecolor{juliaGreen}{rgb}{0.22, 0.596, 0.149}
\definecolor{juliaPurple}{rgb}{0.706, 0.322, 0.804}

\usepackage{tikz}
\usepackage{pgfplots}
\pgfplotsset{compat=newest}
\pgfplotsset{every axis legend/.append style={legend cell align=left, font=\small, draw=none, fill=none}}
\pgfplotsset{every axis/.append style={axis background/.style={fill=white}}}
\pgfplotsset{every axis plot post/.append style={
	smooth,
	line width=2pt,
	line join = round,
	mark = none,
	}
}
\pgfplotsset{
tick label style={font=\footnotesize},
label style={font=\normalsize},
legend style={font=\footnotesize, line width=1pt},
}
\pgfplotsset{every axis/.append style={
thick,
tick style={semithick}}}

\pgfplotsset{
tick label style={font=\footnotesize},
label style={font=\normalsize},
legend style={font=\footnotesize, line width=1pt},
}

\pgfplotsset{every axis/.append style={
thick,
tick style={semithick}}}

\usetikzlibrary{external}
\title{\bf{Robust Control Barrier Functions for Nonlinear Control Systems with Uncertainty: A Duality-based Approach}}

\author{Max H. Cohen, Calin Belta, and Roberto Tron%
\thanks{The authors are with the Department of Mechanical Engineering, Boston University, 110 Cummington Mall, Boston, MA 02215 $\mathtt{\curly{maxcohen,cbelta,tron}@bu.edu}$. M.~Cohen is supported by the NSF GRFP under grant DGE-1840990. C. Belta is supported by the NSF under grant IIS-2024606. R.~Tron is supported by ONR grant N00014-19-1-2571.}
}

\date{}

\begin{document}
\maketitle

\begin{abstract}
    This paper studies the design of controllers that guarantee stability and safety of nonlinear control affine systems with parametric uncertainty in both the drift and control vector fields. To this end, we introduce novel classes of robust control barrier functions (RCBF) and robust control Lyapunov functions (RCLF) that facilitate the synthesis of safety-critical controllers in the presence of parametric uncertainty using quadratic programming. Since the initial bounds on the system uncertainty may be highly conservative, we present a data-driven approach to reducing such bounds using input-output data collected online. In particular, we leverage an integral set-membership identification algorithm that iteratively shrinks the set of possible system parameters online and guarantees stability and safety during learning. The efficacy of the developed approach is illustrated on two numerical examples.
\end{abstract}

\section{Introduction}
Two fundamental concepts in modern nonlinear control theory are (asymptotic) stabilization and safety: requiring a closed-loop system to eventually reach a desired state and requiring a closed-loop system to never do anything ``bad,'' respectively. The former property can often be enforced by constructing a suitable \emph{control Lyapunov function} (CLF)~--~a Lyapunov function candidate whose derivative can be made negative at each state by appropriate control action \cite{SontagSCL89,AmesTAC14}. When the property of safety is formalized using set-theoretic notions (i.e., a system is considered safe if its closed-loop trajectories remain within some prescribed safe set at all times), similar Lyapunov-based techniques can be transposed to design controllers enforcing safety of the closed-loop system. In particular, the concept of a \emph{control barrier function} (CBF) plays a role dual to that of CLFs for safety, allowing one to synthesize control inputs at each state that ensure the desired safe set is forward invariant \cite{AmesTAC17,AmesECC19}. When the underlying system is control affine, CLFs and CBFs facilitate the computation of inputs guaranteeing stability and safety using quadratic programming, which has allowed for safe and stable control of complex nonlinear systems such as autonomous vehicles, bipedal robots, and multi-agent systems (see \cite{AmesECC19} for a survey of applications).

One limitation of traditional quadratic program (QP)-based CLF/CBF controllers is their strong reliance on an accurate system model. Since the mathematical models used for control design are generally a simplification of the true underlying system dynamics, it is essential that any controller take into account model uncertainties stemming from external disturbances, unknown parameters, and other unmodeled dynamics. This paradigm has been well noted in the literature and as a result there are many works that take robust \cite{JankovicAutomatica18,FanCoRL21,TomlinCDC21,GargACC21,JankovicLCSS22,SreenathTAC21}, adaptive \cite{TaylorACC20,LopezLCSS21,CohenACC22}, or data-driven \cite{TaylorL4DC20,TaylorCDC21,SreenathCDC21,EgerstedtTRO21,DhimanTAC21} approaches to accounting for uncertainty in the mathematical models used to generate safe and stable controllers.

Although many approaches have been successfully developed for handling uncertainty in the system drift dynamics or uncertainty stemming from additive disturbances \cite{JankovicAutomatica18,TomlinCDC21,GargACC21,TaylorACC20,LopezLCSS21,CohenACC22}, it is often more challenging to extend such approaches to systems with actuation uncertainty. The main technical challenge in such an extension stems from the difficulty in developing linear constraints on the control input whose satisfaction is sufficient for stability/safety. A popular approach to overcoming this challenge is to derive conic constraints on the control input whose satisfaction is sufficient for safety, which can then be embedded in a \emph{second order cone program} (SOCP) \cite{TaylorCDC21,SreenathCDC21,DhimanTAC21}. SOCPs are convex but generally more computationally intensive to solve than a QP. Other approaches avoid the construction of a SOCP by using the vertex representation a convex set containing the uncertainty as constraints in a QP \cite{FanCoRL21,EgerstedtTRO21}. Although this approach leads to control synthesis using a QP, the number of constraints can grow rapidly in higher dimensions. For example, the vertex representation of an $n$-dimensional hyperrectangle results in $2^n$ constraints, whereas the halfspace representation only results in $2n$ constraints.

In this paper we develop a QP framework for robust stabilization and safety of nonlinear control affine systems with parametric uncertainty in both the system drift and control directions. The key to our approach is to leverage the dual of an auxiliary linear program (LP) to convert bilinear constraints on the control input and uncertain parameters that arise from accounting for the worst-case model uncertainty into linear constraints whose satisfaction is sufficient for stability/safety. These linear constraints essentially allow us to use a halfspace representation of a given (polytopic) uncertainty set, rather than a vertex representation such as in \cite{FanCoRL21,EgerstedtTRO21}, which scales more favorably to higher dimensions. Since the initial bounds on the system uncertainty may be highly conservative, we leverage a data-driven approach in which input-output data collected \emph{online} is used to reduce the bounds on the system uncertainty during run-time while maintaining stability/safety guarantees. This reduction in uncertainty is accomplished using a novel \emph{integral} set-membership identification (SMID) algorithm for continuous-time systems 
that does not require knowledge of the state derivative.
Although nonparametric approaches \cite{TaylorL4DC20,TaylorCDC21,SreenathCDC21,EgerstedtTRO21,DhimanTAC21} account for more general classes of uncertainty, we argue that our parametric approach is more practical --- the \emph{structure} of the dynamics for many relevant systems, such as those in robotics, are well known but often possess uncertainty in parameters such as inertia, friction, damping, etc.

\textit{Contributions}:
The contributions of this paper are threefold. First, we present a novel duality-based approach to develop affine constraints on the control input of systems with parametric uncertainty in both the drift and control vector fields; such constraints can be embedded in a QP to compute robust safe and stabilizing controllers. Second, we outline a novel integral SMID algorithm that learns the true set of possible system parameters using data collected online, and reduces the level of uncertainty in the system model at run-time. Third, we present numerical examples comparing the performance of the developed method with and without the SMID algorithm active and illustrate how our approach can be extended to systems with higher relative degree.

\textit{Organization}:
This paper is organized as follows. Sec. \ref{sec:prelim} contains preliminaries on CBFs and CLFs. Sec. \ref{sec:problem} presents our problem formulation and our assumptions on the system uncertainty. Sec. \ref{sec:robust} presents our duality-based approach to robust stabilization and safety. Sec. \ref{sec:smid} outlines our SMID algorithm for uncertainty reduction. Sec. \ref{sec:sims} contains numerical examples. The paper ends with concluding remarks and directions for future research in Sec. \ref{sec:conclusion}.

\textit{Notation}:
For a continuously differentiable function $h\,:\,\R^n\rightarrow\R$ and a vector field $f\,:\,\R^n\rightarrow\R^n$ we use $L_fh(x)$ to denote the \emph{Lie derivative} of $h$ along $f$. The operator $\norm{\cdot}$ denotes the 2-norm. A continuous function $\alpha\,:\,\R\rightarrow\R$ is an \emph{extended class} $\mathcal{K}$ function if $\alpha(0)=0$ and $\alpha$ is strictly increasing. For $x\in\R^n$, $\text{diag}(x)\in\R^{n\times n}$ returns a diagonal matrix whose elements are the components of $x$. The notation $\mathbf{1}_n\coloneqq[1\cdots1]\T\in\R^n$ stands for a vector of ones.

\section{Preliminaries}\label{sec:prelim}
Consider a nonlinear control affine system of the form
\begin{equation}\label{eq:dyn}
    \dot{x}=f(x) + g(x)u,
\end{equation}
where $x\in\R^n$ is the system state, $u\in\mathcal{U}\subseteq\R^m$ is the control input, $f\,:\,\R^n\rightarrow\R^n$
is a locally Lipschitz vector field modeling the system drift, and $g\,:\,\R^n\rightarrow\R^{n\times m}$ is a locally Lipschitz matrix whose columns capture the control directions. Given a locally Lipschitz feedback control policy $k\,:\,\R^n\rightarrow\R^m$, let $x\,:\,\mathcal{I}\rightarrow\R^n$ be the resulting solution of the closed-loop system \eqref{eq:dyn} under the control signal $u(t)=k(x(t))$ defined on some maximal interval of existence $\mathcal{I}\subseteq\R_{\geq0}$. We say that a closed set $\mathcal{C}\subset\R^n$ is \emph{forward invariant} for the closed-loop system \eqref{eq:dyn} if any solution $x(\cdot)$ starting in $\mathcal{C}$ satisfies $x(t)\in\mathcal{C}$ for all $t\in\mathcal{I}$. As set invariance is often synonymous with \emph{safety}, we refer to sets $\mathcal{C}$ that are forward invariant for \eqref{eq:dyn} as \emph{safe sets}. In this paper, we consider safe sets characterized as the zero superlevel set of a continuously differentiable function $h\,:\,\R^n\rightarrow\R$ as
\begin{equation}\label{eq:C}
    \mathcal{C}\coloneqq\{x\in\R^n\,|\,h(x)\geq0\}.
\end{equation}
The concept of a \emph{control barrier function} (CBF) \cite{AmesTAC17} provides a constructive tool for the design of controllers that render sets of the form \eqref{eq:C} forward invariant and thus safe.
\begin{definition}[\cite{AmesTAC17}]\label{def:CBF}
    A continuously differentiable function $h\,:\,\R^n\rightarrow\R$ is said to be a \emph{control barrier function} (CBF) for \eqref{eq:dyn} on a set $\mathcal{C}$ as in \eqref{eq:C} if there exists an extended class $\mathcal{K}$ function $\alpha$ such that for all $x\in\mathcal{C}$
    \begin{equation}\label{eq:CBF}
        \sup_{u\in\mathcal{U}}\curly{L_fh(x) + L_gh(x)u}\geq -\alpha(h(x)).
    \end{equation}
\end{definition}
Importantly, a CBF $h$ induces a set-valued map $K_{\mathrm{cbf}}(x)$ that associates to each $x\in\mathcal{C}$ a set $K_{\mathrm{cbf}}(x)\subseteq\mathcal{U}$ of control values satisfying the condition from \eqref{eq:CBF} as
\begin{equation}\label{eq:U_cbf}
    K_{\mathrm{cbf}}(x)\coloneqq\curly{u\in\mathcal{U}\,|\,L_fh(x) + L_gh(x)u\geq - \alpha(h(x))}.
\end{equation}
For each $x\in\mathcal{C}$, \eqref{eq:U_cbf} captures affine constraints on the control input, allowing control values satisfying the CBF condition \eqref{eq:CBF} to be computed by solving the quadratic program (QP)

\begin{equation*}\label{eq:CBF-QP}
    \begin{aligned}
        \min_{u\in\mathcal{U}} & \quad \tfrac{1}{2}\|u-k_d(x)\|^2 \\
        \subjectto & \quad L_fh(x) + L_gh(x)u \geq - \alpha(h(x)),
    \end{aligned}
\end{equation*}
where $k_d\,:\,\R^n\rightarrow\R^m$ is any locally Lipschitz nominal control policy. The main result with regards to CBFs is that applying any locally Lipschitz control policy contained in $K_{\mathrm{cbf}}(x)$ to \eqref{eq:dyn} renders $\mathcal{C}$ forward invariant for the closed-loop system \cite{AmesTAC17}. The proofs of many CBF-related results are facilitated by the following comparison lemma.

\begin{lemma}[\cite{EgerstedtLCSS17}]\label{lemma:h}
    Consider a locally Lipschitz extended class $\mathcal{K}$ function $\alpha$ and let $\dot{h}\,:\,[t_1,t_2]\rightarrow \R$ be absolutely continuous. Provided $h(x(t_1))\geq0$ and $\dot{h}(x(t))\geq-\alpha(h(x(t)))$ for almost all $t\in[t_1,t_2]$, then $h(x(t))\geq0$ for all $t\in[t_1,t_2]$.
\end{lemma}
Typically CBFs are used in conjunction with control Lyapunov functions (CLFs) \cite{SontagSCL89,AmesTAC14} to synthesize control policies guaranteeing stability and safety.
\begin{definition}[\cite{AmesECC19}]
    A continuously differentiable positive definite function $V\,:\,\R^n\rightarrow\R_{\geq0}$ is said to be a \emph{control Lyapunov function} (CLF) for \eqref{eq:dyn} on a set $\mathcal{D}$ if there exists a class $\mathcal{K}$ function $\gamma$ such that for all $x\in\mathcal{D}$
    \begin{equation}\label{eq:CLF}
        \inf_{u\in\mathcal{U}}\{L_fV(x) + L_gV(x)u\} \leq -\gamma(V(x)).
    \end{equation}
\end{definition}
Similar to CBFs, the condition in \eqref{eq:CLF} constitutes an affine constraint on the control input, allowing stabilizing control inputs to be computed at any $x\in\mathcal{D}$ by solving a QP \cite{AmesTAC14,AmesTAC17,AmesECC19}.

\section{Problem Formulation}\label{sec:problem}

The primary objective of this paper is to develop CBF-based control policies for systems of the form \eqref{eq:dyn} with \emph{parametric uncertainty} in the vector fields $f,g$. To this end, we consider the following uncertain nonlinear system
\begin{equation}\label{eq:uncertain_dyn}
    \dot{x} = f(x) + g(x)u + \Delta f(x) + \Delta g(x) u,
\end{equation}
where $f,g$ capture known nominal components of the dynamics and $\Delta f\,:\,\R^n\rightarrow\R^n$, $\Delta g\,:\,\R^n\rightarrow\R^{n\times m}$ represent unknown dynamics, assumed to satisfy the following:

\begin{assumption}\label{assumption:lip}
    The uncertain dynamics $\Delta f,\,\Delta g$ can be decomposed as
    \begin{equation*}
        \begin{aligned}
        \Delta f(x)= & \sum_{i=1}^p \Delta f_i(x)\theta_{f_i}=F(x)\theta_f, \\
        \Delta g(x)u= & \sum_{i=1}^m \Delta g_i(x)\theta_{g_i}u_i = G(x)\text{diag}(u)\theta_g
        \end{aligned}
    \end{equation*}
    where $F(x)\coloneqq[\Delta f_1(x)\,\cdots\, \Delta f_p(x)]\in\R^{n\times p}$ and $G(x)\coloneqq[\Delta g_1(x)\,\cdots\,\Delta g_m(x)]\in\R^{n\times m}$ are \emph{known} matrix-valued functions and $\theta_f\coloneqq [\theta_{f_1}\,\cdots\,\theta_{f_p}]\T\in\R^p$, $\theta_g\coloneqq[\theta_{g_1}\,\cdots\,\theta_{g_m}]\T\in\R^m$ are vectors of \emph{unknown} parameters.
\end{assumption}
The above assumption implies that \eqref{eq:uncertain_dyn} can be seen as affine in the parameters, i.e.,
\begin{equation}\label{eq:dyn_lip}
    \dot{x}=f(x) + g(x)u + \varphi(x,u)\theta,
\end{equation}
where $\varphi(x,u)\coloneqq[F(x)\:G(x)\text{diag}(u)]\in\R^{n\times(p+m)}$ is a composite regression matrix and $\theta\coloneqq[\theta_f\T\:\theta_g\T]\T\in\R^{p+m}$ is a composite vector of uncertain parameters. Note that we assume there is a single uncertain parameter associated with each control direction, which holds for many practical systems. Additionally, we assume the unknown parameters belong to a known hyperrectangle $\Theta\subset\R^{p+m}$:

\begin{assumption}\label{assumption:theta}
    There exist known constants $\underline{\theta}_i,\,\overline{\theta}_i\in\R$ for all $i\in\{1,\dots,p+m\}$ and a hyperrectangle $\Theta\coloneqq[\underline{\theta}_1,\overline{\theta}_1]\times\cdots\times[\underline{\theta}_{p+m},\overline{\theta}_{p+m}]\subset\R^{p+m}$ such that $\theta\in\Theta$.
\end{assumption}
Assumption \ref{assumption:theta} implies the set of possible parameters $\Theta$ admits a halfspace representation as $\Theta=\{\theta\in\R^{p+m}\,|\,A\theta \leq b\}$, where $A,b$ capture linear halfspace constraints. Importantly, the above assumptions are satisfied by a variety of practical systems. Examples include robotic systems described by the standard manipulator equations with known bounds on parameters associated with inertia, damping, friction, etc. 
We are now ready to formally state the main problem considered in this paper.
\begin{problem}\label{prob:main}
Given the uncertain system \eqref{eq:dyn_lip} satisfying Assumptions \ref{assumption:lip}-\ref{assumption:theta} and a safe set $\mathcal{C}\subset\R^n$ as in \eqref{eq:C}, find a control policy $u=k(x)$ that guarantees the forward invariance of $\mathcal{C}$ and/or asymptotic stability of the origin.
\end{problem}

\section{A Duality-based Approach to Robust Stability and Safety}\label{sec:robust}
\subsection{Robust Control Barrier Functions}

In this section, we develop a CBF approach that robustly accounts for all possible realizations of the system uncertainty to address Problem \ref{prob:main}. Importantly, we show how this can be accomplished while retaining the traditional QP structure used in CBF approaches by exploiting the dual of a particular linear program (LP). We begin by introducing the notion of a robust CBF for systems of the form \eqref{eq:dyn_lip}.

\begin{definition}\label{def:RCBF}
    A continuously differentiable function $h\,:\,\R^n\rightarrow\R$ is said to be a robust CBF (RCBF) for \eqref{eq:dyn_lip} on a set $\mathcal{C}\subset\R^n$ as in \eqref{eq:C} if there exists an extended class $\mathcal{K}$ function $\alpha$ such that for all $x\in\mathcal{C}$
    \begin{equation}\label{eq:RCBF}
        \sup_{u\in\mathcal{U}}\inf_{\theta\in\Theta} \dot{h}(x,u,\theta) \geq -\alpha(h(x)),
    \end{equation}
    where $\dot{h}(x,u,\theta)=L_fh(x) + L_gh(x)u + L_{\varphi}h(x,u)\theta$.
\end{definition}
Similar to the standard CBF case, let
\begin{equation*}
    \begin{aligned}
        K_{rcbf}(x)\coloneqq & \{u\in\mathcal{U}\,|\,L_fh(x) + L_gh(x)u  \\
        & + \inf_{\theta\in\Theta}L_{\varphi}h(x,u)\theta \geq -\alpha(h(x)) \}
    \end{aligned}
\end{equation*}
be, for each $x\in\mathcal{C}$, the set of control values satisfying the condition from \eqref{eq:RCBF}. The following lemma shows that any locally Lipschitz control policy $k(x)\in K_{rcbf}(x)$ renders $\mathcal{C}$ forward invariant for the closed-loop system.
\begin{lemma}\label{lemma:RCBF}
    If $h$ is a RCBF for \eqref{eq:dyn_lip} over a set $\mathcal{C}$ as in \eqref{eq:C}, $K_{\text{rcbf}}(x)$ is nonempty for each $x\in\mathcal{C}$, and Assumptions \ref{assumption:lip}-\ref{assumption:theta} hold, then any locally Lipschitz control policy $u=k(x)$ satisfying $k(x)\in K_{rcbf}(x)$ for each $x\in\mathcal{C}$ renders $\mathcal{C}$ forward invariant for the closed-loop system.
\end{lemma}
\begin{proof}
    The derivative of $h$ along the closed-loop system is lower bounded as
    \begin{equation*}
        \begin{aligned}
            \dot{h}(x) &= L_fh(x) + L_gh(x)k(x) + L_{\varphi}h(x,k(x))\theta \\
            & \geq L_fh(x) + L_gh(x)k(x) + \inf_{\theta\in\Theta}L_{\varphi}h(x,k(x))\theta \\
            & \geq -\alpha(h(x)).
        \end{aligned}
    \end{equation*}
    Hence, for all $t\in\mathcal{I}$ along the closed-loop system trajectory $x\,:\,\mathcal{I}\rightarrow\R^n$ we have $\dot{h}(x(t))\geq -\alpha(h(x(t)))$ and it follows from Lemma \ref{lemma:h} that $\mathcal{C}$ is forward invariant.
\end{proof}
Although the above lemma demonstrates that the class of CBF from Def. \ref{def:RCBF} provides sufficient conditions for safety, this formulation is not appealing from a control synthesis perspective. In particular, the minimax nature and coupling of control and parameters in Def. \ref{def:RCBF} will lead to bilinear constraints on the control and parameters and thus cannot be directly cast as a QP. To remedy this, note that the inner minimization problem from \eqref{eq:RCBF} can be written as the LP\footnote{Note that $L_{\varphi}h(x,u)$ is an affine function of $u$.}:
\begin{equation}\label{eq:LP}
    \begin{aligned}
        \inf_{{\theta}} & \quad L_{\varphi}h(x,u){\theta} \\
        \subjectto & \quad A{\theta}\leq b.
    \end{aligned}
\end{equation}
The dual of \eqref{eq:LP} is
\begin{equation}\label{eq:LP_dual}
    \begin{aligned}
        \sup_{\mu\leq 0} & \quad b\T\mu \\
        \subjectto & \quad \mu\T A =  L_{\varphi}h(x,u),
    \end{aligned}
\end{equation}
where $\mu$ is the dual variable. In light of \eqref{eq:LP} and \eqref{eq:LP_dual} we show in Theorem \ref{theorem:RCBF} that one can solve the following QP
\begin{equation}\label{eq:RCBF-QP}
    \begin{aligned}
        \min_{u\in\mc{U},\,\mu\leq0} & \quad \tfrac{1}{2}\|u - k_d(x)\|^2 \\
        \subjectto & \quad L_fh(x) + L_gh(x)u + b\T\mu\geq - \alpha(h(x)) \\
        & \quad \mu\T A =  L_{\varphi}h(x,u),  \\
    \end{aligned}
\end{equation}
with decision variables $u$ and $\mu$, to compute a controller satisfying the RCBF conditions from Def. \ref{def:RCBF}.

\begin{theorem}\label{theorem:RCBF}
    Let the assumptions of Lemma \ref{lemma:RCBF} hold. Then any locally Lipschitz solution to \eqref{eq:RCBF-QP}, $u=k(x)$, renders $\mathcal{C}$ forward invariant for the closed-loop system.
\end{theorem}

\begin{proof}
    The RCBF condition \eqref{eq:RCBF} is satisfied at a state $x\in\mathcal{C}$ if the value of the optimization problem
    \begin{equation}\label{eq:opt1}
        \begin{aligned}
            \sup_{u\in\R^m}\inf_{\theta\in\R^{p+m}}& \quad L_fh(x) + L_gh(x)u + L_{\varphi}h(x,u)\theta \\
            \subjectto & \quad u\in\mathcal{U},\;A\theta \leq b,
        \end{aligned}
    \end{equation}
    is greater than or equal to $-\alpha(h(x))$. It follows from the strong duality theorem of LPs \cite[Thm. 4.4]{BertsimasTsitsiklis} that the values of the primal and dual LPs in \eqref{eq:LP} and \eqref{eq:LP_dual}, respectively, are equal, allowing the inner minimization in \eqref{eq:opt1} to be replaced with its dual \eqref{eq:LP_dual} yielding
     \begin{equation}\label{eq:opt2}
        \begin{aligned}
            \sup_{u\in\R^m,\mu\in\R^{2(p+m)}}& \quad L_fh(x) + L_gh(x)u + b\T \mu \\
            \subjectto & \quad u\in\mathcal{U},\;\mu\T A = L_{\varphi}h(x,u),\;\mu \leq 0.
        \end{aligned}
    \end{equation}
    By the strong duality of LPs, the values of the optimization problems in \eqref{eq:opt1} and \eqref{eq:opt2} are equivalent implying that if the optimal value of \eqref{eq:opt2} is greater than or equal to $-\alpha(h(x))$ for a given $x\in\mathcal{C}$, then the resulting input $u$ satisfies \eqref{eq:RCBF}. Embedding the conditions imposed by \eqref{eq:opt2} as constraints in an optimization problem yields the QP in \eqref{eq:RCBF-QP}. Under the presumption that $K_{rcbf}(x)$ is nonempty for each $x\in\mathcal{C}$, the optimal value of \eqref{eq:opt1}, and thus of \eqref{eq:opt2} by strong duality, is greater than or equal to $-\alpha(h(x))$, which implies that \eqref{eq:RCBF-QP} is feasible for each $x\in\mathcal{C}$ and that $k(x)\in K_{rcbf}(x)$ for each $x\in\mathcal{C}$. It then follows from the assumption that that the resulting control policy $u=k(x)$ is locally Lipschitz and Lemma \ref{lemma:RCBF} that such a policy renders $\mathcal{C}$ forward invariant for the closed-loop system, as desired.
\end{proof}

\begin{remark}\label{remark:hocbf}
    For ease of exposition, all results in this section have been stated for relative degree one CBFs. It is possible to extend our approach to high order CBFs \cite{WeiTAC21-hocbf} provided the uncertain parameters satisfy the assumptions made in \cite{CohenACC22}. An example of such an extension is provided in Sec. \ref{sec:hocbf-sim}.
\end{remark}
\begin{remark}
  An alternative way to replacing \eqref{eq:LP} with \eqref{eq:LP_dual} would be to use the fact that, for an LP, the optimum value is achieved at a vertex of the feasible set. Therefore, it is possible to replace the constraint given by \eqref{eq:LP} with an enumeration of constraints obtained by replacing $\theta$ with each corner of the feasible polyhedron $A\theta\leq b$. In general, however, this would result in a number of constraints that grows combinatorially in the number of half spaces in $A\theta\leq b$. Intuitively, this is avoided in \eqref{eq:LP_dual} because the dual variable $\mu$ automatically selects the worst-case corner.
\end{remark}

\subsection{Robust Control Lyapunov Functions}
The duality-based approach developed for robust safety naturally extends to robust stabilization problems using the notion of a robust CLF for systems of the form \eqref{eq:dyn_lip}. For all results in this section we make the following assumption.
\begin{assumption}\label{assumption:eq_pt}
    The uncertain system \eqref{eq:dyn_lip} satisfies $f(0)=0$, which implies that $\varphi(0,0)=0$ and the origin is an equilibrium point of the unforced system.
\end{assumption}
\begin{definition}\label{def:RCLF}
    A continuously differentiable positive definite function $V\,:\,\R^n\rightarrow\R_{\geq0}$ is said to be a Robust CLF (RCLF) for \eqref{eq:dyn_lip} on a set $\mathcal{D}\subseteq\R^n$ if there exists a class $\mathcal{K}$ function $\gamma$ such that for all $x\in\mathcal{D}$
    \begin{equation}\label{eq:RCLF}
        \inf_{u\in\mathcal{U}}\sup_{\theta\in\Theta}\dot{V}(x,u,\theta)\leq -\gamma(V(x)),
    \end{equation}
    where $\dot{V}(x,u,\theta)=L_fV(x) + L_gV(x)u + L_{\varphi}V(x,u)\theta$.
\end{definition}
Now consider the set
\begin{equation*}
    \begin{aligned}
        K_{rclf}(x)\coloneqq & \{u\in\mathcal{U}\,|\, L_fV(x) + L_gV(x)u  \\ &+ \sup_{\theta\in\Theta}L_{\varphi}V(x,u)\theta \leq - \gamma(V(x))\},
    \end{aligned}
\end{equation*}
of all control values satisfying the condition from \eqref{eq:RCLF}. The following lemma shows that any locally Lipschitz controller satisfying the conditions of Def. \ref{def:RCLF} renders the origin asymptotically stable for \eqref{eq:dyn_lip}.
\begin{lemma}\label{lemma:RCLF}
    If $V$ is a RCLF for \eqref{eq:dyn_lip} on a set $\mathcal{D}$ containing the origin, $K_{rclf}(x)$ is nonempty for each $x\in\mathcal{D}$, and Assumptions \ref{assumption:lip}-\ref{assumption:eq_pt} hold, then any locally Lipschitz control policy $u=k(x)$ satisfying $k(x)\in K_{{rclf}}(x)$ for each $x\in\mathcal{D}$ renders the origin asymptotically stable for \eqref{eq:dyn_lip}.
\end{lemma}
\begin{proof}
    The derivative of $V$ along the closed-loop system is upper bounded as
    \begin{equation*}
        \begin{aligned}
            \dot{V}(x) = &  L_fV(x) + L_gV(x)k(x) + L_{\varphi}V(x,k(x))\theta \\
            \leq &  L_fV(x) + L_gV(x)k(x) + \sup_{\theta\in\Theta}L_{\varphi}V(x,k(x))\theta\\
            \leq &  -\gamma(V(x)),
        \end{aligned}
    \end{equation*}
    and asymptotic stability follows from \cite[Thm. 4.1]{Khalil}.
\end{proof}

Following the same duality-based approach as in the previous section we can make the synthesis of robust stabilizing controllers more tractable than as presented in Def. \ref{def:RCLF}. The dual of the LP $\sup_{\theta\in\Theta}L_\varphi V(x,u)\theta$ is given by
\begin{equation}
    \begin{aligned}
        \inf_{\lambda\geq 0}\quad & b\T\lambda  \\
        \subjectto\quad & \lambda\T A = L_\varphi V(x,u),
    \end{aligned}
\end{equation}
where $\lambda$ is the dual variable. This allows to generate inputs satisfying condition \eqref{eq:RCLF} by solving the following QP:
\begin{equation}\label{eq:RCLF-QP}
    \begin{aligned}
        \min_{u\in\mc{U},\,\lambda\geq0} & \quad \tfrac{1}{2}\|u\|^2 \\
        \subjectto & \quad L_fV(x) + L_gV(x)u + b\T\lambda\leq - \gamma(V(x)) \\
        & \quad \lambda\T A = L_{\varphi}V(x,u),
    \end{aligned}
\end{equation}
as shown in the following theorem.

\begin{theorem}\label{theorem:RCLF}
    Let the assumptions of Lemma \ref{lemma:RCLF} hold.
    Then, any locally Lipschitz solution to \eqref{eq:RCLF-QP}, $u=k(x)$, renders the origin asymptotically stable for the closed-loop system.
\end{theorem}
\begin{proof}
    Follows the same steps as that of Theorem \ref{theorem:RCBF}. 
\end{proof}
Provided the sufficient conditions of Theorems \ref{theorem:RCBF} and \ref{theorem:RCLF} are satisfied, inputs enforcing stability and safety can be computed for each $x\in\mathcal{C}$ by taking the solution\footnote{It is also possible to embed both RCBF and RCLF constraints in a single QP; however, we find that in practice better performance is achieved by filtering the RCLF policy through the RCBF QP.} to \eqref{eq:RCLF-QP} as $k_d$ in \eqref{eq:RCBF-QP}.

\section{Online Learning for Uncertainty Reduction}\label{sec:smid}
The previous section demonstrates how to robustly account for system uncertainty to guarantee stability and/or safety; however,  the initial bounds on the system uncertainty may be highly conservative, which could restrict the system from exploring much of the safe set and, as illustrated in Sec. \ref{sec:sims}, could produce controllers that require large amounts of control effort to enforce stability and safety. A more attractive approach is to leverage input-output data generated by the system at run-time in an effort to identify the system uncertainty, which can be used to reduce the conservatism of the approach outlined in the previous section. To this end, we present an integral variant of the SMID algorithm \cite{BoydTAC92} commonly employed in the model predictive control (MPC) literature  \cite{MorariAutomatica14,Lopez} (and recently used in the CBF literature \cite{LopezLCSS21}) to construct the set of possible system parameters that are consistent with the input-output data observed at run time. Since MPC methods typically operate in discrete-time, classical SMID algorithms only require measurements of the system state. When such approaches are used in continuous-time \cite{LopezLCSS21}, such an approach requires measurements or numerical computations of state derivatives, which are generally unavailable or noisy, respectively. Taking inspiration from \cite{DixonIJACSP19}, we outline in this section a SMID algorithm for continuous-time systems that only requires knowledge of the system state and control input.

Following the approach from \cite{DixonIJACSP19}, let $\Delta t\in\R_{>0}$ be the length of an integration window and note that over any finite time interval $[t-\Delta t,t]\in\mathcal{I}$, the Fundamental Theorem of Calculus can be used to
represent \eqref{eq:dyn_lip} as
\begin{equation*}
    \begin{aligned}
        \underbrace{\int_{t-\Delta t}^t\dot{x}(s)ds}_{\Delta x(t)} = &  \underbrace{\int_{t-\Delta t}^{t}f(x(s))ds}_{\mathcal{F}(t)} + \underbrace{\int_{t-\Delta t}^{t}g(x(s))u(s)ds}_{\mathcal{G}(t)}  \\ &+ \underbrace{\int_{t-\Delta t}^{t}\varphi(x(s),u(s))ds}_{\mathcal{S}(t)}\theta.
    \end{aligned}
\end{equation*}
Our goal is now to use the relation
\begin{equation}\label{eq:dyn_int}
    \Delta x(t)=\mathcal{F}(t) + \mathcal{G}(t) + \mathcal{S}(t)\theta \quad \forall t\geq \Delta t,
\end{equation}
to shrink the set of possible parameters $\Theta$ using input-output data collected online.
To this end, let $\mc{H}(t)\coloneqq\curly{\Delta x_j(t),\,\mc{F}_j(t),\,\mc{G}_j(t),\,\mc{S}_j(t) }_{j=1}^{M(t)}$ be a time-varying \emph{history stack} with $M(t)\in\N$ entries, where $\Delta x_j(t)\coloneqq\Delta x(t_i)$, $\mc{F}_j(t)\coloneqq\mc{F}(t_i)$, $\mc{G}_j(t)\coloneqq\mc{G}(t_i)$, and $\mc{S}_j(t)\coloneqq\mc{S}(t_i)$ for some\footnote{The interpretation of the relation $\mc{F}_j(t)\coloneqq\mc{F}(t_i)$ is that $\mathcal{F}_j(t)$ is the value of $\mathcal{F}$ stored in the $j$th slot of the history stack at time $t$, which may have been recorded as some past time $t_i\leq t$.} $t_i\in[\Delta t,t]$. We allow for the number of entries in the history stack $M(t)$ to vary with time since the history stack may be initially empty and redundant data may be removed as new data becomes available \cite{MorariAutomatica14}, and denote by $\mathcal{M}(t)=\{1,\dots,M(t)\}$ the index set of data points at time $t$. Letting $\{t_k\}_{k\in\Z_{\geq0}}$ be a strictly increasing sequence of times with $t_0=0$, consider the corresponding sequence of sets
\begin{equation*}\label{eq:SMID}
    \begin{aligned}
        \Xi_0 = & \Theta \\
        \Xi_k = & \{\theta\in\Xi_{k-1}\,|\,-\varepsilon\mathbf{1}_n\leq \Delta x_j(t_k) - \mathcal{F}_j(t_k) - \mathcal{G}_j(t_k) \\ & - \mathcal{S}_j(t_k)\theta \leq \varepsilon\mathbf{1}_n,\:\forall j\in\mathcal{M}(t_k) \},
    \end{aligned}
\end{equation*}
which is the set of all parameters that approximately satisfy \eqref{eq:dyn_int} for each $j\in\mathcal{M}(t_k)$ with precision\footnote{The constant $\varepsilon$ can be seen as a parameter governing the conservativeness of the identification scheme, which can be used to account for disturbances, noise, unmodeled dynamics, and/or numerical integration errors.} $\varepsilon\in\R_{>0}$. In practice, the set $\Xi_k$ can be computed by solving, for each $i\in\{1,\dots,p+m\}$, the pair of LPs
\begin{equation}\label{eq:SMID_LP1}
    \begin{aligned}
        \underline{\theta}_i^k & =  \argmin_{\theta} \quad \theta _i \\
        \mathrm{s.t.} & \quad \Delta x_j(t_k) - \mathcal{F}_j(t_k) - \mathcal{G}_j(t_k) - \mathcal{S}_j(t_k)\theta \leq \varepsilon\mathbf{1}_n\, \forall j \\
        & \quad \Delta x_j(t_k) - \mathcal{F}_j(t_k) - \mathcal{G}_j(t_k) - \mathcal{S}_j(t_k)\theta \geq -\varepsilon\mathbf{1}_n\, \forall j \\
        & \quad A_{k-1}\theta \leq b_{k-1},
    \end{aligned}
\end{equation}
\begin{equation}\label{eq:SMID_LP2}
    \begin{aligned}
        \overline{\theta}_i^k & =\argmax_{\theta}  \quad \theta _i \\
        \mathrm{s.t.} & \quad \Delta x_j(t_k) - \mathcal{F}_j(t_k) - \mathcal{G}_j(t_k) - \mathcal{S}_j(t_k)\theta \leq \varepsilon\mathbf{1}_n\, \forall j \\
        & \quad \Delta x_j(t_k) - \mathcal{F}_j(t_k) - \mathcal{G}_j(t_k) - \mathcal{S}_j(t_k)\theta \geq -\varepsilon\mathbf{1}_n\, \forall j \\
        & \quad A_{k-1}\theta \leq b_{k-1},
    \end{aligned}
\end{equation}
where $\theta_i$ is the $i$th component of $\theta$ and $A_{k-1}$, $b_{k-1}$ capture the halfspace constraints imposed by $\Xi_{k-1}$. The updated set of possible parameters is then taken as
\begin{equation}\label{eq:Xi}
    \Xi_k=[\underline{\theta}_1^k, \overline{\theta}_1^k]\times\cdots\times [\underline{\theta}_{p+m}^k, \overline{\theta}_{p+m}^k].
\end{equation}

The following result shows that the true parameters always belong to the set of possible parameters generated by the integral SMID scheme.

\begin{lemma}\label{lemma:SMID}
    Provided that Assumptions \ref{assumption:lip}-\ref{assumption:theta} hold and the sequence of sets $\{\Xi_k\}_{k\in\Z_{\geq0}}$ is generated according to \eqref{eq:SMID_LP1}-\eqref{eq:Xi}, then $\Xi_{k}\subseteq \Xi_{k-1}\subseteq\Theta$ and $\theta\in \Xi_k$ for all $k\in\Z_{\geq0}$.
\end{lemma}

\begin{proof}
    The observation that $\Xi_{k}\subseteq \Xi_{k-1}$ for all $k\in\Z_{\geq0}$ follows directly from \eqref{eq:SMID_LP1} and \eqref{eq:SMID_LP2} since the constraint $A_{k-1}\theta\leq b_{k-1}$ ensures that $\underline{\theta}_i^k, \overline{\theta}_i^k\in[\underline{\theta}_i^{k-1}, \overline{\theta}_i^{k-1}]$ for all $i$ implying $[\underline{\theta}_i^{k}, \overline{\theta}_i^{k}] \subseteq[\underline{\theta}_i^{k-1}, \overline{\theta}_i^{k-1}]$ for all $i$. It then follows from \eqref{eq:Xi} and $\Xi_0=\Theta$ that $\Xi_{k}\subseteq \Xi_{k-1}\subseteq\Theta$ for all $k\in\Z_{\geq0}$. Our goal is now to show that $\theta\in\Xi_{k-1}\implies \theta\in\Xi_{k}$. For any $k\in\Z_{\geq0}$, relation \eqref{eq:dyn_int}
    implies that $\theta$ belongs to the set
    \[
    H_k=\{\theta\in\R^{p+m}\,|\,\Delta x_j(t_k) - \mathcal{F}_j(t_k) - \mathcal{G}_k(t_k) - \mathcal{S}_j(t_k)\theta=0\}
    \]
    for all $j\in\mathcal{M}(t_k)$. Additionally, for any $k\in\Z_{\geq0}$ the constraints in \eqref{eq:SMID_LP1}-\eqref{eq:SMID_LP2} ensure that $\Xi_k\subset H_{k}^-\cap H_k^+$, where
    \begin{equation*}
        \begin{aligned}
            H_k^- = & \{\theta\in\R^{p+m}\,|\,\Delta x_j(t_k) - \mathcal{F}_j(t_k) - \mathcal{G}_k(t_k)  \\ & - \mathcal{S}_j(t_k)\theta \geq  -\varepsilon\mathbf{1}_n\} \\
            H_k^+ = & \{\theta\in\R^{p+m}\,|\,\Delta x_j(t_k) - \mathcal{F}_j(t_k) - \mathcal{G}_k(t_k) \\ & - \mathcal{S}_j(t_k)\theta\leq \varepsilon\mathbf{1}_n\},
        \end{aligned}
    \end{equation*}
    for all $j\in\mathcal{M}(t_k)$. It then follows from $\theta\in H_k$ and $H_k\subset H_{k}^-\cap H_k^+$
    that $\theta\in H_{k}^-\cap H_k^+$. The last constraint in \eqref{eq:SMID_LP1}-\eqref{eq:SMID_LP2} ensures that $\Xi_k\subset H_{k}^-\cap H_k^+\cap \Xi_{k-1}$, which implies that $\theta\in\Xi_k$ as long as $\theta\in\Xi_{k-1}$. Since $\theta\in\Xi_0$ it inductively follows from $\theta\in\Xi_{k-1}\implies \theta\in\Xi_{k}$ for all $k\in\Z_{\geq0}$ that $\theta\in \Xi_k$ for all $k\in\Z_{\geq0}$.
\end{proof}

The following propositions demonstrate that if $h$ and $V$ are a RCBF and RCLF, respectively, for \eqref{eq:dyn_lip} with respect to the original parameter set $\Theta$, then they remain so for the parameter sets generated by the SMID algorithm.
\begin{proposition}\label{proposition:RCBF_SMID}
    Let $h$ be a RCBF for \eqref{eq:dyn_lip} on a set $\mathcal{C}\subset\R^n$ in the sense that there exists an extended class $\mathcal{K}$ function $\alpha$ such that \eqref{eq:RCBF} holds for all $x\in\mathcal{C}$. Provided the assumptions of Lemma \ref{lemma:SMID} hold, then
    \begin{equation*}
        \sup_{u\in\mathcal{U}}\inf_{\theta\in\Xi_k}\dot{h}(x,u,\theta) \geq -\alpha(h(x)),
    \end{equation*}
    for all $x\in\mathcal{C}$ and all $k\in\Z_{\geq0}$.
\end{proposition}
\begin{proof}
    Let $\theta^*_k\in\Xi_k$ be the solution to the LP $\inf_{\theta\in\Xi_{k}}L_{\varphi}h(x,u)\theta$
    for some fixed $(x,u)$. Since $\Xi_{k+1}\subseteq\Xi_k$ by Lemma \ref{lemma:SMID} one of the following holds: either (i) $\theta^*_k\in\Xi_{k+1}$ or (ii) $\theta^*_k\in\Xi_{k}\backslash\Xi_{k+1}$. For case (i), if the infimum is achieved over the set $\Xi_{k+1}$, then $\theta^*_k$ would also be an optimal solution to the LP $\inf_{\theta\in\Xi_{k+1}}L_{\varphi}h(x,u)\theta$ and
    \begin{equation*}
        \inf_{\theta\in\Xi_{k+1}}L_{\varphi}h(x,u)\theta = \inf_{\theta\in\Xi_{k}}L_{\varphi}h(x,u)\theta.
    \end{equation*}
    For case (ii) if $\theta^*_k\in\Xi_{k}\backslash\Xi_{k+1}$, then necessarily
    \begin{equation*}
        \inf_{\theta\in\Xi_{k+1}}L_{\varphi}h(x,u)\theta \geq \inf_{\theta\in\Xi_{k}}L_{\varphi}h(x,u)\theta,
    \end{equation*}
    otherwise the infimum would have been achieved over $\Xi_{k+1}$ since $\Xi_k\supseteq \Xi_{k+1}$. Thus, since the RCBF condition \eqref{eq:RCBF} holds over $\Theta$ and $\Xi_{k}\subseteq\Theta$ for all $k\in\Z_{\geq0}$ by Lemma \ref{lemma:SMID}, we have
    \begin{equation*}
        \inf_{\theta\in\Xi_{k}}L_{\varphi}h(x,u)\theta \geq \inf_{\theta\in\Theta}L_{\varphi}h(x,u)\theta,
    \end{equation*}
    for all $k\in\Z_{\geq0}$. The preceding argument implies
    \begin{equation*}
        \begin{aligned}
            \sup_{u\in\mathcal{U}}\inf_{\theta\in\Xi_k}\dot{h}(x,u,\theta) \geq \sup_{u\in\mathcal{U}}\inf_{\theta\in\Theta}\dot{h}(x,u,\theta)  \geq -\alpha(h(x)),
        \end{aligned}
    \end{equation*}
    for all $x\in\mathcal{C}$ and $k\in\Z_{\geq0}$, as desired.
\end{proof}

\begin{proposition}\label{proposition:RCLF_SMID}
    Let $V$ be a RCLF for \eqref{eq:dyn_lip} on a set $\mathcal{D}\subseteq\R^n$ in the sense that there exists a class $\mathcal{K}$ function $\gamma$ such that \eqref{eq:RCLF} holds for all $x\in\mathcal{D}$. Provided the assumptions of Lemma \ref{lemma:SMID} hold, then
    \begin{equation*}
        \inf_{u\in\mathcal{U}}\sup_{\theta\in\Xi_k}\dot{V}(x,u,\theta) \leq -\gamma(V(x)),
    \end{equation*}
    for all $x\in\mathcal{D}$ and all $k\in\Z_{\geq0}$.
\end{proposition}

\begin{proof}
    The proof parallels that of Proposition \ref{proposition:RCBF_SMID}.
\end{proof}

\begin{remark}\label{remark:switched}
    Each uncertainty set generated by the SMID algorithm induces a different control policy and hence a different closed-loop system. Thus, as the uncertainty set is updated over time, the original system \eqref{eq:dyn_lip} becomes a switched system with switching instances taking place whenever the uncertainty set is updated. Since the derivative of the RCLF along each subsystem is bounded by the same negative definite term, the RCLF serves as a \emph{common Lyapunov function} \cite[Ch. 2]{Liberzon-switched-systems}, thereby preserving stability under arbitrary switching.
\end{remark}

\section{Numerical Examples}\label{sec:sims}
\paragraph*{Nonlinear System}
We first consider the scenario from \cite{JankovicAutomatica18,CohenAutomatica21}, which involves a two-dimensional system of the form \eqref{eq:dyn} with $f(x)=[\theta_1x_1 + \theta_2x_2\;\theta_3x_1^3]\T\in\R^2$ and $g(x)=[0\;\theta_4x_2]\T\in\R^2$,
where $\theta_1=-0.6$, $\theta_2=-1$, $\theta_3=1$, $\theta_4=1$ are the uncertain parameters. This system can be recast in the form of \eqref{eq:dyn_lip} by defining $f(x)=[0\;0]\T$, $g(x)=[0\;0]\T$
\begin{equation*}
    F(x) =
    \begin{bmatrix}
        x_1 & x_2 & 0 \\ 0 & 0 & x_1^3
    \end{bmatrix},
    \quad
    G(x) =
    \begin{bmatrix}
        0 \\ x_2
    \end{bmatrix},
\end{equation*}
with $\theta_f=[\theta_1\;\theta_2\;\theta_3]\T$ and $\theta_g=\theta_4$. The uncertain parameters are assumed to lie in the set $\Theta=[-1.2,-0.2]\times[-2,-0.1]\times[0.5,1.4]\times[0.8,1.2].$
The objective is to regulate the system to the origin while remaining in a set $\mathcal{C}\subset\R^2$ characterized as in \eqref{eq:C} with $h(x) = 1 - x_1 - x_2^2$. The regulation objective is achieved by considering the RCLF candidate $V(x) = \tfrac{1}{4}x_1^4 + \tfrac{1}{2}x_2^2$ with $\gamma(s)=\tfrac{1}{2}s$ and the safety objective is achieved by considering the RCBF candidate with $h$ as above and $\alpha(s)=s^3$. Given a RCLF, RCBF, and uncertainty set $\Theta$, one can form a QP as noted after Theorem \ref{theorem:RCLF} to generate a closed-loop control policy that guarantees stability and safety provided the sufficient conditions of Theorems \ref{theorem:RCBF} and \ref{theorem:RCLF} are satisfied. To illustrate the impact of the integral SMID procedure, simulations are run with and without SMID active, the results of which are provided in Fig. \ref{fig:nonlinear_traj}-\ref{fig:nonlinear_SMID}. The parameters associated with the SMID simulation are $\Delta t=0.3$, $\varepsilon=0.1$, $M=20$. The $M$ data points in LPs \eqref{eq:SMID_LP1} and \eqref{eq:SMID_LP2} are collected using a moving window approach, where the $M$ most recent data points are used to update the uncertainty set. As illustrated in Fig. \ref{fig:nonlinear_traj} the trajectory under the RCLF-RCBF-QP achieves the stabilization and safety objective with and without SMID; however, the trajectory without any parameter identification is significantly more conservative and is unable to approach the boundary of the safe set. In contrast, the trajectory with SMID is able to approach the boundary of the safe set as more data about the system becomes available. In particular, both trajectories follow an identical path up until $t=\Delta t$, at which point the set of possible parameters is updated, causing the blue curve (SMID) to deviate from the orange curve (no SMID) in Fig. \ref{fig:nonlinear_traj}. In fact, even after the first SMID update the blue curve closely resembles the purple curve, which corresponds to the trajectory under a CBF-QP with perfect model knowledge. Although the parameters have not been exactly identified by the end of the simulation (see Fig. \ref{fig:nonlinear_SMID}), the modest reduction in uncertainty offered by the SMID approach greatly reduces the conservatism of the purely robust approach.

\begin{figure}
    \centering
    \vspace{3mm}
    \includegraphics[]{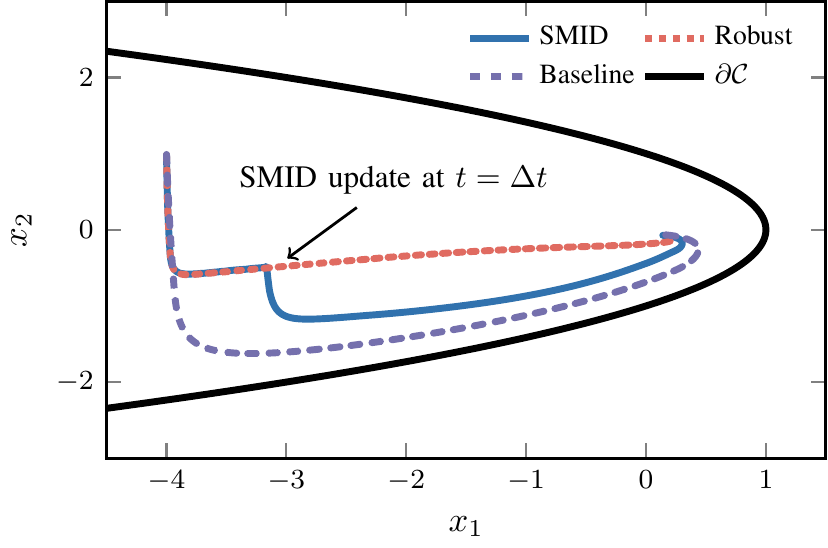}
    \vspace{-4mm}
    \caption{Trajectory of the nonlinear system under various controllers. The solid blue curve depicts the trajectory with SMID, the dotted orange curve depicts the trajectory without SMID, the purple curve illustrates the trajectory under a standard CBF-QP with exact model knowledge, and the black curve denotes the boundary of the safe set.\vspace{-4mm}}
    \label{fig:nonlinear_traj}
\end{figure}

\begin{figure*}
    \centering
    \vspace{3mm}
    \includegraphics[]{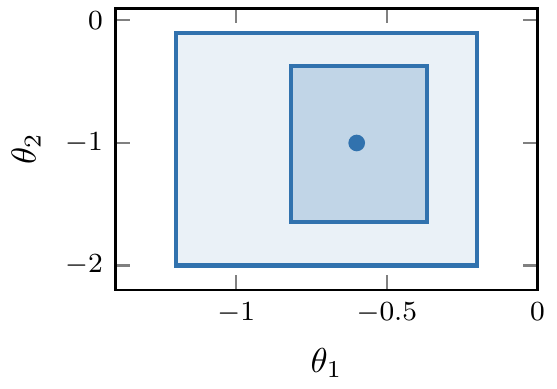}
    \hfill
    \includegraphics[]{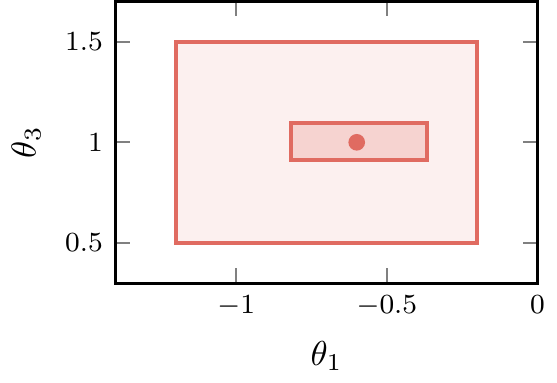}
    \hfill
    \includegraphics[]{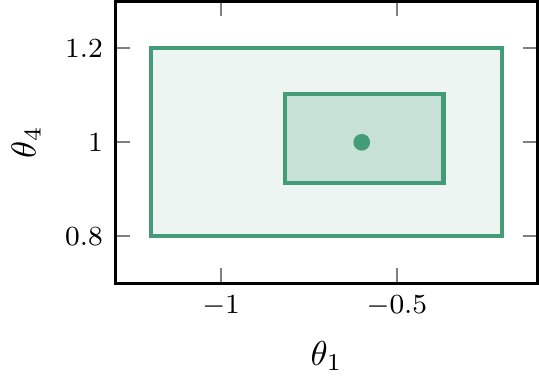}
    \vspace{-4mm}
    \caption{Set-based estimate of the uncertain parameters for the nonlinear system. From left to right, the plots illustrate the uncertainty set $\Theta$ projected onto the $\theta_1\times \theta_2$, $\theta_1\times \theta_3$, and $\theta_1\times \theta_4$ axes, respectively. In each plot the pale rectangle represents the original uncertainty set, the dark rectangle represents the final uncertainty set generated by the SMID algorithm, and the dot represents the true values of the parameters.\vspace{-4mm}}
    \label{fig:nonlinear_SMID}
\end{figure*}

\paragraph*{Robotic Navigation}\label{sec:hocbf-sim}
We now consider a robotic navigation task as in \cite{CohenACC22} and demonstrate how to incorporate high order CBFs (HOCBFs) \cite{WeiTAC21-hocbf} into the developed framework. The robot is modeled as a planar double integrator with uncertain mass and friction effects of the form \eqref{eq:dyn} as
\begin{equation*}
    \underbrace{
    \begin{bmatrix}
        \dot{x}_1 \\ \dot{x}_2 \\ \dot{x}_3 \\ \dot{x}_4
    \end{bmatrix}
    }_{\dot{x}}
    =
    \underbrace{
    \begin{bmatrix}
        x_3 \\ x_4 \\ -\tfrac{\mathrm{c}_1}{\mathrm{m}} x_3 \\ -\tfrac{\mathrm{c}_2}{\mathrm{m}}x_4
    \end{bmatrix}
    }_{f(x)}
    +
    \underbrace{
    \begin{bmatrix}
        0 & 0 \\ 0 & 0 \\ \tfrac{1}{\mathrm{m}} & 0 \\ 0 & \tfrac{1}{\mathrm{m}}
    \end{bmatrix}
    }_{g(x)}
    \underbrace{
    \begin{bmatrix}
        u_1 \\ u_2
    \end{bmatrix}
    }_{u}
\end{equation*}
where $[x_1\;x_2]\T\in\R^2$ represents the robot's position, $[x_3\;x_4]\T\in\R^2$ its velocity, $u\in\R^2$ its acceleration input, $[\mathrm{c}_1\;\mathrm{c}_2]\T\in\R^2$ are uncertain friction coefficients, and $\mathrm{m}\in\R_{>0}$ is an uncertain mass. This system can be represented as in \eqref{eq:dyn_lip} by defining $f(x)=[x_3\;x_4\;0\;0]\T$, $g(x)=0_{4\times 2}$, $\theta_f=[\tfrac{\mathrm{c}_1}{\mathrm{m}}\;\tfrac{\mathrm{c}_2}{\mathrm{m}}]\T$, $\theta_g=[\tfrac{1}{\mathrm{m}}\;\tfrac{1}{\mathrm{m}}]\T$, and
\begin{equation*}
    F(x) =
    \begin{bmatrix}
        0 & 0 \\ 0 & 0 \\ -x_3 & 0 \\ 0 & -x_4
    \end{bmatrix},
    \quad
    G(x) =
    \begin{bmatrix}
        0 & 0 \\ 0 & 0 \\ 1 & 0 \\ 0 & 1
    \end{bmatrix}.
\end{equation*}
The objective is to drive the robot to the origin while avoiding a circular obstacle of radius $r\in\R_{>0}$ centered at $[x_o\;y_o]\T\in\R^2$. The candidate safe set can be described as the zero superlevel set of $h(x) = (x_1 - x_o)^2 + (x_2 - y_o)^2 - r^2.$ However, note that $\nabla h(x)=[2x_1\;2x_2\;0\;0]$ and thus $L_gh(x),\,L_Fh(x),\,L_Gh(x)\equiv 0$, which implies that the relative degree \cite[Def. 13.2]{Khalil} of $h$ with respect to $u$ is larger than one and $h$ is not a CBF for this particular system. One way to overcome this difficulty is to leverage HOCBFs \cite{WeiTAC21-hocbf}, which employ a backstepping-like methodology to systematically inject higher order terms into a CBF candidate. As noted in Remark \ref{remark:hocbf}, HOCBFs can be leveraged for the uncertain system \eqref{eq:dyn_lip} provided the uncertain terms satisfy the conditions posed in \cite[Assumption 1]{CohenACC22}, which requires the relative degree of $h$ with respect to the uncertain parameters to be the same as that of the control input. By computing the second derivative of $h$ along the system dynamics one can verify such an assumption holds for this system and candidate safe set.

To further demonstrate the advantage of reducing the level of uncertainty online, we simulate the double integrator under a robust HOCBF-based policy with and without the integral SMID algorithm running. For each simulation, the uncertain parameters are assumed to lie in the set $\Theta=[0,5]\times[0,5]\times[0.1,2]\times[0.1,2]$ and all extended class $\mathcal{K}$ functions used in the HOCBF constraints are chosen as $\alpha(s)=s^3$ (see \cite{CohenACC22,WeiTAC21-hocbf} for further details on the formulation of HOCBF constraints). The stabilization objective is achieved by considering the same CLF candidate used in \cite{CohenACC22} and the controller ultimately applied to the system is computed by filtering the solution to the RCLF-QP \eqref{eq:RCLF-QP} through a robust HOCBF-QP. The parameters for the SMID algorithm are chosen as $M=20$, $\Delta t=0.1$, and $\varepsilon=1$, where data is recorded using the same technique as in the previous example. The trajectory of the robot's position with and without the SMID algorithm is illustrated in Fig. \ref{fig:dbl_int_traj}, where each trajectory is shown to satisfy the stability and safety objective. Although the trajectories appear very similar, the controller without SMID generates this trajectory with significantly more control effort (see Fig. \ref{fig:double_int_control}). In fact, within the first second of the simulation such a controller requires control effort that is an order of magnitude higher than that of the controller that reduces the uncertainty online to avoid collision with the obstacle. 

\begin{figure}
    \centering
    \includegraphics[]{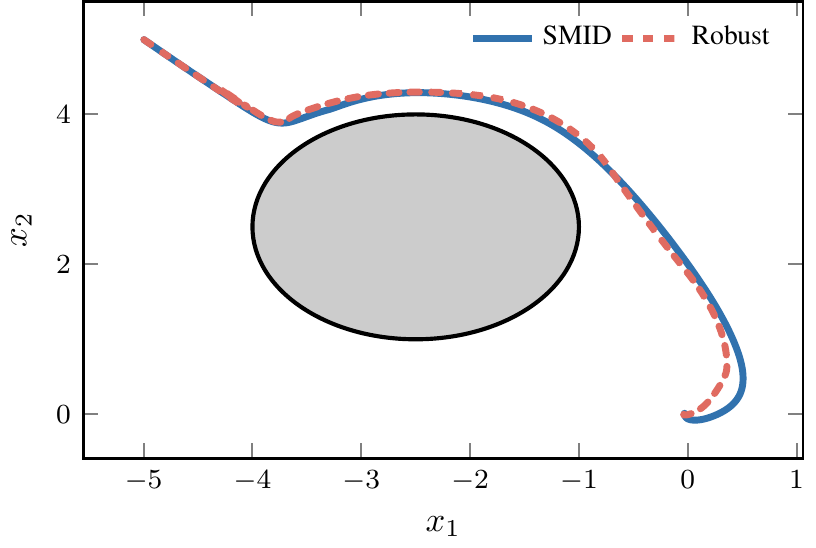}
    \vspace{-4mm}
    \caption{Evolution of the double integrator's position with the SMID algorithm active (blue) and inactive (orange). The gray disk denotes an obstacle of radius $r=1.5$ centered at $x_o=-2.5$, $y_o=2.5$.}
    \label{fig:dbl_int_traj}
\end{figure}

\begin{figure}
    \centering
    \includegraphics[]{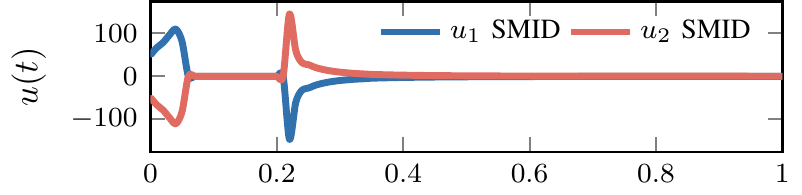}
    \includegraphics[]{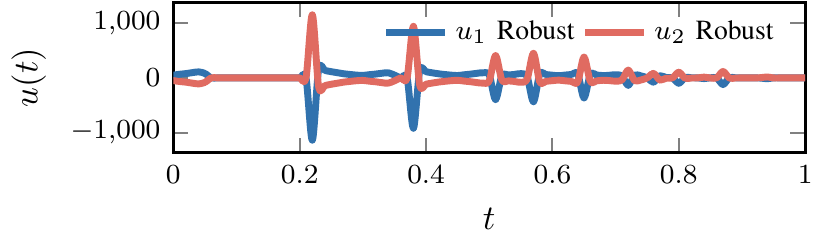}
    \vspace{-4mm}
    \caption{Evolution of the control input for the double integrator with the SMID algorithm active (top) and inactive (bottom). \vspace{-8mm}}
    \label{fig:double_int_control}
\end{figure}

\section{Conclusions}\label{sec:conclusion}
This paper introduced a methodology for robust stabilization and safety of nonlinear control systems in the presence of parametric uncertainty in both the drift and control vector fields. Crucial to this approach are a class of robust CBF and CLF that facilitate the computation of safe and stable control inputs using quadratic programming even when uncertain terms appear alongside the control input. The key insight enabling this approach was that the dual of an auxiliary LP can be used to convert bilinear constraints on the control and parameters into linear constraints that can be embedded within a QP. This robust approach was then combined with data-driven techniques in the form of a novel integral SMID algorithm that allows for the level of uncertainty to be reduced online while maintaining stability and safety guarantees. Potential directions for future research include an investigation into feasibility of the proposed QPs.

\bibliographystyle{ieeetr}
\bibliography{%
biblio/barrier,%
biblio/books,%
biblio/adaptive,%
biblio/hybrid,%
biblio/mpc,%
biblio/nonlinear%
}

\end{document}